\documentclass[12pt,leqno]{article}
\usepackage{amssymb,amsfonts,amsthm,amsmath,latexsym}
\newtheorem{thm}{Theorem}[section]
\newtheorem{prop}[thm]{Proposition}
\newtheorem{lem}[thm]{Lemma}
\newtheorem{cor}[thm]{Corollary}
\theoremstyle{remark}
\newtheorem{rem}[thm]{Remark}

\newcommand{\ZZ}{\mathbb{Z}}

\newcommand{\calC}{\mathcal{C}}
\newcommand{\calD}{\mathcal{D}}

\DeclareMathOperator{\wt}{wt}

\begin{document}

\title{Note on the residue codes of 
self-dual $\ZZ_4$-codes having large minimum Lee weights}

\author{
Masaaki Harada\thanks{
Research Center for Pure and Applied Mathematics, 
Graduate School of Information Sciences, 
Tohoku University, Sendai 980--8579, Japan.
email: mharada@m.tohoku.ac.jp.
This work 
was partially carried out
at Yamagata University.}
}

\maketitle

\begin{abstract}
It is shown that
the residue code of a self-dual $\ZZ_4$-code of length $24k$
(resp.\ $24k+8$)
and minimum Lee weight $8k+4 \text{ or }8k+2$
(resp.\ $8k+8 \text{ or }8k+6$)
is a binary extremal doubly even self-dual code 
for every positive integer $k$.
A number of new self-dual $\ZZ_4$-codes of length $24$
and minimum Lee weight $10$ are constructed using the 
above characterization.
These codes are Type~I $\ZZ_4$-codes
having the largest minimum Lee weight and the largest
Euclidean weight among all Type~I $\ZZ_4$-codes of that length.
In addition, new extremal Type~II $\ZZ_4$-codes of length
$56$ are found.
\end{abstract}

\section{Introduction}
\label{Sec:Intro}
Self-dual codes are an important class of 
(linear) codes\footnote{All codes in this note are 
linear unless otherwise noted.} for both
theoretical and practical reasons.
It is a fundamental problem to classify self-dual codes
of modest length
and 
determine the largest minimum weight among self-dual codes
of that length.
Among self-dual $\ZZ_k$-codes, self-dual $\ZZ_4$-codes
have been widely studied because such codes
have nice applications to unimodular lattices and 
(non-linear) binary codes,
where $\ZZ_k$ denotes the ring of integers modulo $k$
and $k$ is a positive integer with $k \ge 2$.
It is well known that
the Nordstorm--Robinson and Kerdock codes, which
are some of the best known non-linear binary codes, 
can be constructed as the Gray 
images of some $\ZZ_4$-codes~\cite{HKCSS}.
We emphasize that
the Nordstorm--Robinson code can be constructed
as the Gray image of the unique self-dual
$\ZZ_4$-code of length $8$ and minimum Lee weight $6$.
In this note, 
we pay attention to the minimum Lee weight 
from the viewpoint of a connection with
the minimum distance of binary (non-linear) codes obtained as
the Gray images.
%
Rains~\cite{Rains00} gave upper bounds on the minimum Lee weights
$d_L(\calC)$ of 
self-dual $\ZZ_4$-codes $\calC$ of length $n$.
For even lengths $n=24k+\ell$, 
the upper bounds are given as $d_L(\calC) \le 8k+g(\ell)$,
where $g(\ell)$ is given by the following table:

\begin{center}
\begin{tabular}{c|cccccccccccc}
$\ell$ & 0&2&4&6&8&10 & 12&14&16&18&20&22 \\ 
\hline
$g(\ell)$&4&2&4&4&8&4& 4&6&8&8&8&8\\
\end{tabular}
\end{center}
%
%

In this note, we study residue codes of
self-dual $\ZZ_4$-codes having large minimum Lee weights.
According to the above upper bounds,
the minimum Lee weights of 
self-dual $\ZZ_4$-codes of lengths $24k$ and $24k+8$
are at most $8k+4$ and $8k+8$, respectively.
It is shown that
the residue code of a self-dual $\ZZ_4$-code of length $24k$
and minimum Lee weight $8k+4 \text{ or }8k+2$
is a binary extremal doubly even self-dual code of length $24k$
for every positive integer $k$.
It is also shown that 
the residue code of a self-dual $\ZZ_4$-code of length $24k+8$
and minimum Lee weight $8k+8 \text{ or }8k+6$
is a binary extremal doubly even self-dual code of length $24k+8$.
As a consequence, we show that
the minimum Lee weight of a self-dual $\ZZ_4$-code of 
length $24k$ (resp.\ $24k+8$)
is at most $8k$ (resp.\ $8k+4$)
for every integer $k \ge 154$ (resp.\ $k \ge 159$).
A number of new self-dual $\ZZ_4$-codes of length $24$
and minimum Lee weight $10$ are constructed using the 
above characterization.
It is also shown that any self-dual $\ZZ_4$-code of length $24$
and minimum Lee weight $10$ is 
a Type~I $\ZZ_4$-code having the largest minimum Lee weight and the 
largest Euclidean weight among all Type~I $\ZZ_4$-codes of that length.
Some self-dual $\ZZ_4$-codes of length $n$ and
minimum Lee weight $d_L$ are also constructed
for the cases $(n,d_L)=(32,14),(48,18),(56,18)$.
The case $(n,d_L)=(56,18)$
gives two new extremal Type~II $\ZZ_4$-codes of length $56$.
Finally, we give a certain characterization of 
binary self-dual codes containing the residue codes 
of self-dual $\ZZ_4$-codes for some other lengths.

All computer calculations in this note
were done by {\sc Magma}~\cite{Magma}.

\section{Preliminaries}

\subsection{Self-dual $\ZZ_4$-codes}

Let $\ZZ_4\ (=\{0,1,2,3\})$ denote the ring of integers
modulo $4$.
A {\em $\ZZ_{4}$-code} $\calC$ of length $n$ 
is a $\ZZ_{4}$-submodule of $\ZZ_{4}^n$.
Two $\ZZ_4$-codes are {\em equivalent} if one can be obtained from the
other by permuting the coordinates and (if necessary) changing
the signs of certain coordinates.
The {\em dual code} ${\calC}^\perp$ of $\calC$ is defined as
${\calC}^\perp = \{ x \in \ZZ_{4}^n \mid x \cdot y = 0$ for 
all $y \in \calC\}$,
where $x \cdot y$ is the standard inner product.
A $\ZZ_4$-code $\calC$ is {\em self-dual} if $\calC=\calC^\perp.$
The 
{\em Hamming weight} $\wt_H(x)$,
{\em Lee weight} $\wt_L(x)$ and
{\em Euclidean weight} $\wt_E(x)$ 
of a codeword $x$ of $\calC$ are defined as
$n_1(x)+n_2(x)+n_3(x)$,
$n_1(x)+2n_2(x)+n_3(x)$ and
$n_1(x)+4n_2(x)+n_3(x)$, respectively, 
where $n_i(x)$ is the number of components of $x$ which
are equal to $i$.
The {\em minimum Lee weight} $d_L(\calC)$ 
(resp.\ {\em minimum Euclidean weight} $d_E(\calC)$)
of $\calC$
is the smallest Lee (resp.\ Euclidean) weight among
all non-zero codewords of $\calC$.
The {\em residue code} $\calC^{(1)}$ of 
$\calC$ is the binary code defined as
$\calC^{(1)}= \{ c \pmod 2\ |\ c \in \calC \}$.
If $\calC$ is a self-dual $\ZZ_4$-code, then 
$\calC^{(1)}$ is doubly even~\cite{Z4-CS}.

The following characterization of the minimum Lee weights
is useful.

\begin{lem}[Rains~\cite{Rains}]
\label{lem1}
Let $\calC$ be a self-dual $\ZZ_4$-code.
Then $d({\calC}^{(1)}) \le d_L(\calC) \le 2d({{\calC}^{(1)}}^\perp)$.
\end{lem}

%

The {\em Gray map} $\phi$ is defined as a map from
$\ZZ_4^n$ to $\ZZ_2^{2n}$
mapping $(x_1,\ldots,x_n)$ to
$(\varphi(x_1),\ldots,\varphi(x_n))$,
where $\varphi(0)=(0,0)$, $\varphi(1)=(0,1)$, $\varphi(2)=(1,1)$ and
$\varphi(3)=(1,0)$.
The Gray image $\phi(\calC)$ of a $\ZZ_4$-code $\calC$ need not be
linear.
Let $\calC$ be a self-dual $\ZZ_4$-code of length $n$ 
and minimum Lee weight $d_L(\calC)$.
Then the Gray image $\phi(\calC)$
has parameters $(2n,2^{n},d_L(\calC))$ (as a non-linear code).

A self-dual $\ZZ_4$-code which has the property that all
Euclidean weights are divisible by eight,
is called {\em Type~II}.
A self-dual $\ZZ_4$-code which is not Type~II, is 
called {\em Type~I}.
A Type~II $\ZZ_4$-code of length $n$ exists if and 
only if $n \equiv 0 \pmod 8$,
while a Type~I $\ZZ_4$-code exists for every length.
It was shown in~\cite{Z4-BSBM} that the minimum Euclidean 
weight $d_E(\calC)$ of a Type~II $\ZZ_4$-code $\calC$ 
of length $n$ is bounded by
$d_E(\calC) \le 8 \lfloor \frac{n}{24} \rfloor +8$.
A Type~II $\ZZ_4$-code meeting this bound is called 
{\em extremal}.
It was also shown in~\cite{Rains-Sloane98} that
the minimum Euclidean 
weight $d_E(\calC)$ of a Type~I $\ZZ_4$-code $\calC$ 
of length $n$ is bounded by
$d_E(\calC) \le 8 \lfloor \frac{n}{24} \rfloor +8$
if $n \not\equiv 23 \pmod{24}$, and
$d_E(\calC) \le 8 \lfloor \frac{n}{24} \rfloor +12$
if $n \equiv 23 \pmod{24}$.

\subsection{Binary self-dual codes, covering radii and shadows}
A binary code $C$ is called 
{\em self-dual} if $C = C^{\perp}$, where
$C^{\perp}$ is the dual code of $C$ under the
standard inner product.
Two binary self-dual codes $C$ and $C'$ are {\em equivalent}, 
denoted $C \cong C'$,
if one can be obtained from the other by permuting the coordinates.
A binary self-dual code $C$ is {\em doubly even} if all
codewords of $C$ have weight divisible by four, and {\em
singly even} if there is at least one codeword of weight 
congruent to $2$ modulo $4$.
It is known that a binary self-dual code of length $n$ exists 
if and only if  $n$ is even, and
a binary doubly even self-dual code of length $n$
exists if and only if $n \equiv 0 \pmod 8$.
The minimum weight $d(C)$ of a binary self-dual code $C$ of length $n$
is bounded by
$d(C)  \le 4 \lfloor{\frac {n}{24}} \rfloor + 6$ 
if $n \equiv 22 \pmod {24}$, 
$d  \le 4  \lfloor{\frac {n}{24}} \rfloor + 4$ 
otherwise~\cite{MS73} and~\cite{Rains98}.
A binary self-dual code meeting the bound is called  {\em extremal}.

The {\em covering radius} $R(C)$ of a binary code $C$ 
is the smallest integer $R$
such that spheres of radius $R$ around codewords of $C$ cover the
space $\ZZ_2^n$.
The covering radius is a basic and important geometric parameter of a code.
The covering radius is the same as the
largest weight of all the coset leaders of the code 
(see~\cite{A-P}).
The following bound is known as the Delsarte bound
(see~\cite[Theorem 1]{A-P}).

\begin{lem}\label{lem:Del}
Let $C$ be a binary code.
Then $R(C) \le \#\{i > 0 \mid B_i \ne 0\}$,
where $B_i$ is the number of vectors
of weight $i$ in $C^\perp$.
\end{lem}

Let $C$ be a binary singly even self-dual code and let $C_0$ denote the
subcode of codewords having weight congruent to $0$ modulo $4$.
Then $C_0$ is
a subcode of codimension $1$. The {\em shadow} $S$ of $C$ is
defined to be $C_0^\perp \setminus C$. 
Shadows were introduced by Conway and Sloane~\cite{C-S}, in order
to
provide restrictions on the weight enumerators of singly even
self-dual codes. 
A binary self-dual code meeting the following bound is called  
{\em s-extremal}.

\begin{lem}[Bachoc and Gaborit~\cite{BG04}]
\label{lem:sextremal}
Let $C$ be a binary self-dual code of length $n$
and let $S$ be the shadow of $C$.
Let $d(C)$ and $d(S)$ denote the minimum weights
of $C$ and $S$, respectively.
Then
$
d(S) \le \frac{n}{2} + 4 -2d(C), 
$
except in the case that
$n \equiv 22 \pmod{24}$ and $d(C)=
4  \lfloor{\frac {n}{24}} \rfloor + 6$,
where $d(S) = \frac{n}{2} + 8 -2d(C)$. 
\end{lem}

We end this section by stating the following lemma,
which is obtained from~\cite[Theorems~2.1 and 2.2]{MST}.

\begin{lem}\label{lem:SD}
Let $C$ be a binary self-orthogonal code of
length $n$.  
\begin{itemize}
\item[\rm (i)]
If $n$ is even, then there is a binary 
self-dual code containing $C$.
\item[\rm (ii)]
If $n \equiv 0 \pmod 8$ and $C$ is doubly even which is not self-dual,
then there is a binary doubly even
self-dual code containing $C$, and
there is a binary singly even
self-dual code containing $C$.
\end{itemize}
\end{lem}

\section{Characterization of the residue codes for lengths $24k$ and $24k+8$}

\subsection{Length $24k$}


As described in Section~\ref{Sec:Intro}, 
the minimum Lee weight of a self-dual $\ZZ_4$-code of length $24k$
is at most $8k+4$.
In this subsection, we consider self-dual $\ZZ_4$-codes of length $24k$
and minimum Lee weight $8k+4 \text{ or }8k+2$.

\begin{thm}\label{thm:24k}
Let $\calC$  be a self-dual $\ZZ_4$-code of length $24k$.
Suppose that the minimum Lee weight of $\calC$
is $8k+4 \text{ or }8k+2$.
Then $\calC^{(1)}$ is a binary extremal doubly even self-dual
code of length $24k$.
\end{thm}
\begin{proof}
Since ${\calC}^{(1)}$ is doubly even, 
by Lemma~\ref{lem:SD},
there is  a binary doubly even self-dual code $C$ satisfying that
${\calC}^{(1)} \subseteq C \subseteq {{\calC}^{(1)}}^\perp$.
Since $\calC$ has minimum Lee weight $8k+4$ (resp.\ $8k+2$),
by Lemma~\ref{lem1}, ${{\calC}^{(1)}}^\perp$
has minimum weight at least $4k+2$ (resp.\ $4k+1$).
Hence, $C$ is extremal.

%
Now consider the covering radius $R(C)$ of $C$.
By Lemma~\ref{lem:Del},
$R(C) \le 4k$.
Hence, if $C \subsetneq {{\calC}^{(1)}}^\perp$, then 
the minimum weight of ${{\calC}^{(1)}}^\perp$ is at most $4k$,
which is a contradiction.
Therefore, $C=\calC^{(1)}$.
\end{proof}

%


\begin{rem}
Recently, the nonexistence of a self-dual $\ZZ_4$-code of
length $36$ and minimum Lee weight $16$ has been shown 
in~\cite{Kiermaier}.
Although this result can be directly obtained by the bound 
in~\cite{Rains00}, which is given in Section~\ref{Sec:Intro},
the approach in~\cite{Kiermaier} can be
generalized to the following alternative proof of the above theorem.
Suppose that $\calC^{(1)}$ is not self-dual.
Since ${\calC}^{(1)}$ is doubly even, 
by Lemma~\ref{lem:SD},
there is a binary singly even self-dual code $C$
satisfying
\[
{\calC}^{(1)} \subseteq C_0 \subsetneq C \subsetneq C_0^\perp
\subseteq {{\calC}^{(1)}}^\perp, 
\]
where
$C_0$ denotes the doubly even subcode of $C$.
By Lemma~\ref{lem1}, ${{\calC}^{(1)}}^\perp$
has minimum weight at least $4k+1$.
By~\cite[Theorem~5]{Rains98}, 
$C$ has minimum weight $4k+2$.
By Lemma~\ref{lem:sextremal},
the minimum weight of the shadow of a binary
singly even self-dual $[24k,12k,4k+2]$ code is at most $4k$,
which is a contradiction.
Hence,  ${\calC}^{(1)}$ is self-dual, that is, 
${\calC}^{(1)}$ is extremal.
This completes the alternative proof.
\end{rem}

\begin{rem}
For lengths up to $24$, 
optimal self-dual $\ZZ_4$-codes with respect to the minimum
Hamming and Lee weights were widely studied in~\cite{Rains}.
At length $24$, 
the above theorem follows from~\cite[Theorem~2 and Corollary~5]{Rains}. 
\end{rem}


For length $24k$, the only known binary extremal doubly even
self-dual codes are the extended Golay code $G_{24}$ and the
extended quadratic residue code $QR_{48}$ of length $48$.
The existence of a binary extremal doubly even self-dual
code of length $72$ is a long-standing open question.
The above theorem gives a certain characterization of the
existence of a binary extremal doubly even self-dual
code of length $24k$.
In addition, 
there is no  binary extremal doubly even self-dual code
of length $24k$ for $k \ge 154$~\cite{Zhang}.
Hence, we immediately have the following:

\begin{cor}
The minimum Lee weight of a self-dual $\ZZ_4$-code of length $24k$
is at most $8k$ for every integer $k \ge 154$.
\end{cor}


\subsection{Length $24k+8$}

As described in Section~\ref{Sec:Intro}, 
the minimum Lee weight of a self-dual $\ZZ_4$-code of length $24k+8$
is at most $8k+8$.
In this subsection, we consider self-dual $\ZZ_4$-codes of length $24k+8$
and minimum Lee weight $8k+8 \text{ or }8k+6$.

\begin{thm}\label{thm:24k+8}
Let $\calC$  be a self-dual $\ZZ_4$-code of length $24k+8$.
Suppose that the minimum Lee weight of $\calC$
is $8k+8 \text{ or }8k+6$.
Then $\calC^{(1)}$ is a binary extremal doubly even self-dual
code of length $24k+8$.
\end{thm}
\begin{proof}
Suppose that $\calC^{(1)}$ is not self-dual.
Since ${\calC}^{(1)}$ is doubly even, 
by Lemma~\ref{lem:SD},
there is a binary singly even self-dual code $C$
satisfying 
\[
{\calC}^{(1)} \subseteq C_0 \subsetneq C \subsetneq C_0^\perp
\subseteq {{\calC}^{(1)}}^\perp, 
\]
where
$C_0$ denotes the doubly even subcode of $C$.
By Lemma~\ref{lem1}, ${{\calC}^{(1)}}^\perp$
has minimum weight at least $4k+3$.
Hence, $C$ has minimum weight $4k+4$.
By Lemma~\ref{lem:sextremal},
the minimum weight of the shadow of a binary
singly even self-dual $[24k+8,12k+4,4k+4]$ code 
is at most $4k$,
which is a contradiction.
Hence,  ${\calC}^{(1)}$ is self-dual, that is, 
${\calC}^{(1)}$ is extremal.
\end{proof}
\begin{rem}
\begin{itemize}
\item[\rm (i)]
The case that the minimum Lee weight $d_L(\calC)$ is 
$8k+8$ follows immediately 
from~\cite[Theorem~1]{Rains00}.
\item[\rm (ii)]
The above theorem can be proved by 
a similar argument to the proof of Theorem~\ref{thm:24k}.
\end{itemize}
\end{rem}

\begin{rem}\label{rem}
Rains~\cite[p.~148]{Rains00} pointed out that
by the linear programming
$d_L(\calC) \le 8k+6$ for $k \le 4$.
\end{rem}

It is known that there is a binary 
extremal doubly even self-dual code of length $24k+8$ for $k \le 4$.
In addition, 
since there is no  binary extremal doubly even self-dual code
of length $24k+8$ for $k \ge 159$~\cite{Zhang},
we immediately have the following:

\begin{cor}
The minimum Lee weight of a self-dual $\ZZ_4$-code of length $24k+8$
is at most $8k+4$ for every integer $k \ge 159$.
\end{cor}

\section{Self-dual $\ZZ_4$-codes having large minimum Lee weights}

By using the characterizations of the residue codes, which are given in the
previous section, 
a number of self-dual $\ZZ_4$-codes having large minimum Lee weights
are constructed in this section.

\subsection{Double circulant and four-negacirculant codes}

Throughout this note, let $A^T$ denote the transpose of a matrix $A$
and let $I_k$ denote the identity matrix of order $k$.
An  $n \times n$ matrix is 
{\em circulant} and {\em negacirculant} if
it has the following form:
\[
\left( \begin{array}{ccccc}
r_0     &r_1     & \cdots &r_{n-2}&r_{n-1} \\
cr_{n-1}&r_0     & \cdots &r_{n-3}&r_{n-2} \\
cr_{n-2}&cr_{n-1}& \ddots &r_{n-4}&r_{n-3} \\
\vdots  & \vdots &\ddots& \ddots & \vdots \\
cr_1    &cr_2    & \cdots&cr_{n-1}&r_0
\end{array}
\right),
\]
where $c=1$ and $-1$, respectively.
A $\ZZ_4$-code with generator matrix of the form:
\begin{equation}
\label{eq:bDCC}
\left(
\begin{array}{ccc@{}c}
\quad & {\Large I_{n}} & \quad &
\begin{array}{cccc}
\alpha & \beta  & \cdots & \beta  \\
\gamma & {}     & {}     &{} \\
\vdots & {}     & R      &{} \\
\gamma & {}     &{}      &{} \\
\end{array}
\end{array}
\right)
\end{equation}
is called a {\em bordered double circulant} $\ZZ_4$-code of length $2n$,
where $R$ is an $(n-1) \times (n-1)$ circulant matrix and
$\alpha, \beta, \gamma \in \ZZ_4$.
A $\ZZ_4$-code with generator matrix of the form:
\begin{equation} \label{eq:4nega}
\left(
\begin{array}{ccc@{}c}
\quad & {\Large I_{2n}} & \quad &
\begin{array}{cc}
A & B \\
-B^T & A^T
\end{array}
\end{array}
\right)
\end{equation}
is called a {\em four-negacirculant} $\ZZ_4$-code of length $4n$,
where $A$ and $B$ are $n \times n$ negacirculant matrices.

\begin{table}[thb]
\caption{Bordered double circulant self-dual $\ZZ_4$-codes}
\label{Tab:DCC}
\begin{center}
{\small
\begin{tabular}{c|c|l|c|c|c}
\noalign{\hrule height1pt}
Length & Code & \multicolumn{1}{c|}{First row of $R$} 
  & $(\alpha,\beta,\gamma)$ & Type& $d_L$\\
\hline
24 & $\calD_{24,1}$ & (13103303222) &	$(0,1,1)$ &I&10\\
   & $\calD_{24,2}$ & (01130332322) &	$(0,1,1)$ &I&10\\
   & $\calD_{24,3}$ & (31030001332) &	$(0,1,1)$ &I&10\\
\hline
32 & $\calD_{32}$   & (002210100233312) &	$(0,1,1)$ &II&14\\
\hline
48 & $\calD_{48}$ & (11303312013230033212110) & $(0,1,1)$ &II& $18$ \\
\hline
56 & $\calD_{56,1}$& (022000202022112232101111011)&$(2,1,1)$&II&18\\
   & $\calD_{56,2}$& (002202002002312010101111011)&$(0,1,1)$&I&18\\
\noalign{\hrule height1pt}
\end{tabular}
}
\end{center}
\end{table}

By considering bordered double circulant codes
and  four-negacirculant codes, 
we found 
self-dual $\ZZ_4$-codes of length $24k$ and minimum Lee weight 
$8k+2$ $(k=1,2)$ and
self-dual $\ZZ_4$-codes of length $32$ and 
minimum Lee weight $14$.
These codes were found under the condition that 
the residue codes are
binary extremal doubly even self-dual codes, by
Theorems~\ref{thm:24k} and \ref{thm:24k+8}.
Self-dual $\ZZ_4$-codes of length $56$ and minimum Lee weight $18$
were also found.

For bordered double circulant codes,
the first rows of $R$ and $(\alpha,\beta,\gamma)$ in 
(\ref{eq:bDCC}) are listed in Table~\ref{Tab:DCC}.
For four-negacirculant codes,  
the first rows of $A$ and $B$ in (\ref{eq:4nega})
are listed in Table~\ref{Tab:4nega}.
The minimum Lee weights $d_L$ determined by {\sc Magma}
are also listed.
The $5$th column in 
both tables indicates whether the given code is
Type~I or Type~II.

\begin{table}[thb]
\caption{Four-negacirculant self-dual $\ZZ_4$-codes}
\label{Tab:4nega}
\begin{center}
{\small
\begin{tabular}{c|c|l|l|c|c}
\noalign{\hrule height1pt}
Length & Code 
& \multicolumn{1}{c|}{First row of $A$} 
& \multicolumn{1}{c|}{First row of $B$} & Type & $d_L$\\
\hline
32 & $\calC_{32}$ & (22312012) & (03113022) &II &14 \\
\hline
56 & $\calC_{56}$ & (11130213112212)   & (30101110001000) & II &18 \\
\noalign{\hrule height1pt}
\end{tabular}
}
\end{center}
\end{table}

\subsection{Length 24}

For length $24$,
there are $13$ self-dual $\ZZ_4$-codes having 
minimum Lee weight $12$, up to equivalence~\cite[Theorem~11]{Rains}.
Note that these self-dual $\ZZ_4$-codes 
are extremal Type~II $\ZZ_4$-codes~\cite[Theorem~9]{Rains}.
Hence, the largest minimum Lee weight among all Type~I $\ZZ_4$-codes
of length $24$ is $10$.
In this subsection,  we consider self-dual $\ZZ_4$-codes having 
minimum Lee weight $10$.

\begin{prop}
Let $\calC$ be a self-dual $\ZZ_4$-code of length $24$ and
minimum Lee weight $10$.
Then $\calC$ is a Type~I $\ZZ_4$-code having minimum Euclidean
weight $12$.
\end{prop}
\begin{proof}
Let $x$ be a codeword $x$ of $\calC$ with $\wt_L(x)=10$.
Then
\[
(n_1(x)+n_3(x),n_2(x))= (10,0), (8,1),  (6,2), (4,3), (2,4), (0,5).
\]
By Theorem~\ref{thm:24k}, 
$\calC^{(1)} \cong G_{24}$. 
Thus, 
$n_1(x)+n_3(x) = 8$ or $n_1(x)+n_3(x)=0$.
In addition, if $n_1(x)+n_3(x)=0$, then
$n_2(x) \equiv 0 \pmod 4$ with $n_2(x) \ge 8$.
This gives
\[
(n_1(x)+n_3(x),n_2(x))= (8,1).
\]
Hence, $\wt_E(x)=12$.
Therefore, $\calC$ is a Type~I $\ZZ_4$-code having minimum Euclidean
weight $12$.
\end{proof}

\begin{rem}
The largest minimum Euclidean weight among all Type~I $\ZZ_4$-codes
of length $24$ is $12$.
\end{rem}

We use the following method in order to verify that given two
$\ZZ_4$-codes are inequivalent (see~\cite{GH97}).
Let $C$ be a self-dual $\ZZ_4$-code of length $n$.  
Let $M_t=(m_{ij})$ be the $A_t \times n$ matrix with rows 
composed of the codewords $x$ with $\wt_H(x)=t$ in $C$, where $A_t$ denotes 
the number of such codewords.
For an integer $k$ ($1 \le k \le n$), 
let $n_t(j_1,\ldots,j_k)$ be the number of $r$ $(1 \le r \le A_t)$ 
such that all $m_{rj_1},\ldots,m_{rj_k}$ are nonzero
for $1 \le j_1 < \ldots < j_k \le n$.  
We consider the set
\[
S_{t,k} = \{ n_t(j_1,\ldots,j_k) \mid \text{ for any distinct $k$ columns } 
j_1,\ldots,j_k\ \}.
\]

In~\cite{GH97}, the authors claimed that there are
two inequivalent bordered double circulant Type I
$\ZZ_4$-codes of length $24$ and minimum Lee weight $10$.
Unfortunately, this is not true.  
In fact, the number of such codes should be three not two.  
The codes $\calD_{24,i}$ $(i=1,2,3)$ given in Table~\ref{Tab:DCC}
are bordered double circulant Type I
$\ZZ_4$-codes of length $24$ and minimum Lee weight $10$.
In Table~\ref{Tab:D24},
we list ${\cal S}_k=(\max(S_{9,k}),\min(S_{9,k}),\# S_{9,k})$ $(k=1,2,3,4)$ 
for the codes.
This table shows that the three codes $\calD_{24,1},\calD_{24,2},\calD_{24,3}$
are inequivalent.

\begin{prop}
There are three inequivalent bordered double circulant Type I
$\ZZ_4$-codes of length $24$ and minimum Lee weight $10$.
\end{prop}

\begin{table}[thb]
\caption{${\cal S}_1$, ${\cal S}_2$, ${\cal S}_3$, ${\cal S}_4$ for 
$\calD_{24,1}$, $\calD_{24,2}$, $\calD_{24,3}$}
\label{Tab:D24}
\begin{center}
{\small
\begin{tabular}{c|c|c|c|c}
\noalign{\hrule height1pt}
Code & 
${\cal S}_1$&${\cal S}_2$&${\cal S}_3$&${\cal S}_4$\\
\hline
$\calD_{24,1}$ 
  &$( 352, 256, 2 )$&$( 128, 0, 5 )$&$( 48, 0, 11 )$&$( 20, 0, 11 )$\\
$\calD_{24,2}$ 
  &$( 352, 256, 2 )$&$( 128, 0, 5 )$&$( 48, 0, 11 )$&$( 18, 0, 10 )$\\
$\calD_{24,3}$ 
  &$( 352, 256, 2 )$&$( 128, 0, 5 )$&$( 48, 0, 11 )$&$( 16, 0, 9 )$\\
\noalign{\hrule height1pt}
\end{tabular}
}
\end{center}
\end{table}


For a given binary doubly even code $C$ of dimension $k$,
there are $2^{\frac{k(k+1)}{2}}$
self-dual $\ZZ_4$-codes $\calC$ with $\calC^{(1)}=C$, and
an explicit method for construction of these $2^{\frac{k(k+1)}{2}}$
self-dual $\ZZ_4$-codes $\calC$ with $\calC^{(1)}=C$
was given in~\cite[Section~3]{Z4-PLF}.
In our case,
there are $2^{78}$ self-dual $\ZZ_4$-codes $\calC$ with 
$\calC^{(1)}=G_{24}$, and 
it seems infeasible to find all such codes.
Using the above method, 
we tried to construct many self-dual $\ZZ_4$-codes.
Then we stopped our search 
after we found $57$ self-dual $\ZZ_4$-codes 
having minimum Lee weight $10$ satisfying that
the $57$ codes and the three codes in Table~\ref{Tab:D24}
have distinct $S_{9,k}$ $(k=1,2,3,4)$.
Hence, we have the following proposition.
%

\begin{prop}
There are at least $60$ inequivalent self-dual $\ZZ_4$-codes
of length $24$ and minimum Lee weight $10$.
\end{prop}

We denote the new codes by 
$\calC_{24,i}$ $(i=1,2,\ldots,57)$.
In Figure~\ref{Fig:24}, 
we list generator matrices for $\calC_{24,i}$, where
we consider generator matrices in standard form
$(\ I_{12}\ , \ M_i\ )$ and only $12$ rows in $M_i$ are listed,
to save space.

\subsection{Lengths 32, 48 and 56}

The extended lifted quadratic residue $\ZZ_4$-code 
${\mathcal{QR}}_{32}$ and the Reed--Muller $\ZZ_4$-code 
${\mathcal{QRM}}(2,5)$, which are given 
in~\cite[Table I]{Z4-BSBM}, 
are self-dual $\ZZ_4$-codes of length $32$ and minimum Lee weight $14$.
Both codes are extremal Type~II $\ZZ_4$-codes~\cite{Z4-BSBM}.
It is known that
${\mathcal{QR}}_{32}^{(1)}$ (resp.\ ${\mathcal{QRM}}(2,5)^{(1)}$)
is the extended quadratic residue code $QR_{32}$ 
(resp.\ a second-order the Reed--Muller code $RM(2,5)$) of length $32$,
which 
is a binary extremal doubly even self-dual code.
The largest minimum Lee weight among bordered double circulant
self-dual $\ZZ_4$-codes is listed in the table in~\cite{KW08}
for length $8n$ $(n=1,2,\ldots,8)$.
According to the table, the largest minimum Lee weight for
length $32$ is $14$.
The code $\calD_{32}$ in Table~\ref{Tab:4nega} 
is a Type~II $\ZZ_4$-code of length $32$ 
and minimum Lee weight $14$, which gives an explicit
example of such codes.
In addition, the code $\calC_{32}$ in Table~\ref{Tab:4nega} 
is a Type~II $\ZZ_4$-code of length $32$ 
and minimum Lee weight $14$.
We verified by {\sc Magma} that
$\calC_{32}^{(1)} \cong \calD_{32}^{(1)} \cong QR_{32}$.
It is unknown whether the three  codes
are equivalent or not.
There are five inequivalent binary extremal doubly even
self-dual codes of length $32$, two of which are
$QR_{32}$ and $RM(2,5)$ (see~\cite[Table IV]{RS-Handbook}).
It is worthwhile to determine whether there is a 
self-dual $\ZZ_4$-code $\calC$ having minimum Lee weight $14$
with $\calC^{(1)} \cong C$ for each $C$ of the remaining three codes.


The extended lifted quadratic residue $\ZZ_4$-code
${\mathcal{QR}}_{48}$ of length $48$ is a self-dual
$\ZZ_4$-code having minimum Lee weight $18$,
which is an extremal Type~II $\ZZ_4$-code.
This is the only known self-dual
$\ZZ_4$-code of length $48$ and minimum Lee weight at least $18$.
Of course, ${\mathcal{QR}}_{48}^{(1)}$ is $QR_{48}$.
According to the table in~\cite{KW08}, the largest minimum Lee weight 
among bordered double circulant self-dual $\ZZ_4$-codes 
of length $48$ is $18$.
The code $\calD_{48}$ in Table~\ref{Tab:DCC} gives an explicit
example of such codes.
It is unknown whether $\calD_{48}$
is equivalent to ${\mathcal{QR}}_{48}$ or not.


At length $56$,
under the condition that
the residue code is a binary extremal doubly even self-dual code,
we tried to construct a self-dual $\ZZ_4$-code having  
minimum Lee weight $20$ or $22$,
but our search failed to do this.
In this process, however, extremal Type~II $\ZZ_4$-codes were found.
The code $\calC_{56}$ in Table~\ref{Tab:4nega} is a 
Type~II $\ZZ_4$-code of length $56$ and minimum Lee weight $18$.
Hence, $\calC_{56}$ is extremal.
According to the table in~\cite{KW08}, the largest minimum Lee weight 
among bordered double circulant self-dual $\ZZ_4$-codes 
of length $56$ is $18$.
The codes $\calD_{56,1}$ and $\calD_{56,2}$ in Table~\ref{Tab:DCC} 
give explicit examples of such codes.
Since $\calD_{56,1}$ is Type~II, 
$\calD_{56,1}$ is extremal.
We verified by {\sc Magma} that 
$\calD_{56,2}$ has minimum Euclidean weight $20$.
We verified by {\sc Magma} that $\calC_{56}^{(1)}$
and $\calD_{56,1}^{(1)}$ have automorphism groups of orders
$28$ and $54$, respectively.
This shows that $\calC_{56}$ and $\calD_{56,1}$ are inequivalent.
An extremal Type~II $\ZZ_4$-code
of length $56$ given in~\cite{HZ456}
has the residue code of dimension $14$.
Hence, we have the following:

\begin{prop}
There are at least three inequivalent extremal Type~II $\ZZ_4$-codes
of length $56$.
\end{prop}

It is unknown whether there is a
self-dual $\ZZ_4$-code having minimum Lee weight $20,22$ or not for length $56$.

At length 80, the minimum Lee weight of
the extended lifted quadratic residue $\ZZ_4$-code 
was determined in~\cite{KW12} as 26.
It is unknown whether there is a
self-dual $\ZZ_4$-code having minimum Lee weight $28,30$ or not.

\section{Characterization of the residue codes for other lengths}

Finally, in this section, we give a certain characterization of 
binary self-dual codes containing the residue codes 
$\calC^{(1)}$ of
self-dual $\ZZ_4$-codes $\calC$ of length $24k+\alpha$
for $\alpha=2,4,6,10,14,16,18,20,22$.

\begin{prop}
Let $\calC$  be a self-dual $\ZZ_4$-code of length $24k+\alpha$
and minimum Lee weight $8k+\beta$, where
$(\alpha,\beta)=
(2,2),
(4,4),
(6,4),
(10,4)$.
Then any binary self-dual code $C$ containing 
$\calC^{(1)}$ is an
s-extremal self-dual code having minimum weight $4k+2$.
\end{prop}
\begin{proof}
Since all cases are similar,  we only give the details for
the case $(\alpha,\beta)=(6,4)$.
By Lemma~\ref{lem:SD},
there is  a binary self-dual code $C$ satisfying
\[
{\calC}^{(1)} \subseteq C_0 \subsetneq C \subsetneq C_0^\perp
\subseteq {{\calC}^{(1)}}^\perp, 
\]
where
$C_0$ denotes the doubly even subcode of $C$.
By Lemma~\ref{lem1}, ${{\calC}^{(1)}}^\perp$
has minimum weight at least $4k+2$.
Hence, $C$ has minimum weight $4k+2$ or $4k+4$.

Suppose that $C$ has minimum weight $4k+4$.
By Lemma~\ref{lem:sextremal},
the minimum weight of the shadow $C_0^\perp \setminus C$
of $C$ is at most $4k-1$,
which contradicts the minimum weight of ${{\calC}^{(1)}}^\perp$.
Now, suppose that $C$ has minimum weight $4k+2$.
The weight of every vector of the shadow $C_0^\perp \setminus C$
is congruent to $3$ modulo $4$~\cite{C-S}.
Since $C_0^\perp$ has minimum weight at least $4k+2$,
the shadow has minimum weight at least $4k+3$.
By Lemma~\ref{lem:sextremal},
the minimum weight of the shadow $C_0^\perp \setminus C$
of $C$ is at most $4k+3$.
Hence, $C$ is $s$-extremal.
\end{proof}

The situations in the following proposition are slightly
different to that in the above proposition.
However, a similar argument to the proof of the above proposition
establishes the following proposition, and their proofs are omitted.

\begin{prop}
Let $\calC$  be a self-dual $\ZZ_4$-code of length $24k+\alpha$
and minimum Lee weight $8k+\beta$.
Let $C$ be a binary self-dual code containing $\calC^{(1)}$.
\begin{itemize}
\item[\rm (i)]
Suppose that $(\alpha,\beta)=
(14,6),
(18,8),
(20,8)$.
Then $C$ is an s-extremal self-dual code having  minimum weight $4k+4$.
\item[\rm (ii)]
Suppose that $(\alpha,\beta)=(16,8)$.
If $C$ is singly even, then $C$ is an
s-extremal self-dual code having minimum weight $4k+4$.
If $C$ is doubly even, then $C$ is extremal.
\item[\rm (iii)]
Suppose that  $(\alpha,\beta)=(22,8)$.
Then
$C$ is an s-extremal self-dual code having minimum weight 
$4k+4$ or $4k+6$.
\end{itemize}
\end{prop}

%
%

\bigskip
\noindent {\bf Acknowledgment.}
This work is supported by JSPS KAKENHI Grant Number 23340021.


\begin{figure}[p]
\centering
{\footnotesize
\begin{tabular}{ll}
$M_{1}$:&
301203221111 131321121202 031330112300 023333033010 020111103321 301010131221\\&
313322212031 330331002332 213211120231 320120311112 230012013313 132223130321,\\
$M_{2}$:&
123021003313 313321321222 231310132322 001313013232 002331103321 321012113201\\&
133322232033 330113222132 211011322231 102122333130 010212011313 110223112301,\\
$M_{3}$:&
123023021313 331121101022 011132310120 221313213030 220111121323 323232313203\\&
113302210033 330331220112 011011302231 120100113310 012012231311 130223110303,\\
$M_{4}$:&
323023003133 131321121222 213310130302 003333033212 002331303303 303230313023\\&
311320032013 110333200132 213033122231 100302111312 212012213311 112203310303,\\
$M_{5}$:&
103203003333 333321101000 031130132300 203333211212 220111303323 303230333001\\&
111320012031 330113002130 013033320033 122122133332 032232013311 132023112101,\\
$M_{6}$:&
101201201331 333321101000 031130132300 203333211210 222113101321 301232131003\\&
113322210033 330113002132 011031122031 122122133332 030230211313 130021310101,\\
$M_{7}$:&
103223003113 131121301220 013132130320 021311031012 200311123323 301032111001\\&
331102210233 332333002130 233213320211 322100131110 032210233333 310021332123,\\
$M_{8}$:&
321223001111 131101103222 013112332322 023311231032 222113123121 301032111001\\&
113320012211 130311022110 011211302231 320122111312 210210233133 132203310301,\\
$M_{9}$:&
123223001111 113121103220 211332312302 023111213212 002331301321 323030313223\\&
131120032011 310311002110 033013120231 122100331330 012212011133 330003332103,\\
$M_{10}$:&
321023201133 111103103202 231112312120 223133031212 002133323323 303012311223\\&
311122230011 132131202112 213013300031 302300133330 210010231311 332221330121,\\
$M_{11}$:&
121021023131 313321321022 213332132100 003311211210 200331101121 121030311021\\&
311102030211 330331220330 031213302031 122100113130 030230033331 112201110323,\\
$M_{12}$:&
123223203113 111121121022 033112132322 203133213012 020313321103 321012111223\\&
133120032231 332113000130 213033120031 302100313132 232032231333 132021312303,\\
$M_{13}$:&
103203203333 331321103222 011332312300 003311213232 020313101121 303232111003\\&
311120232033 330311220130 031031122033 120100313110 212010013113 312223312321,\\
$M_{14}$:&
121201001313 111301123020 233130330302 023131031230 222131123101 101210311003\\&
133320010231 330113022132 231231302231 102300313312 212032231331 130021312303,\\
$M_{15}$:&
323023203111 313103321200 031112110302 003133013210 202111103103 323212133201\\&
331322230233 310331002312 033231100213 320322131310 010212013113 312201110123,\\
$M_{16}$:&
103021201131 113323103002 031332110300 201113231210 200311321323 101232333201\\&
333300012233 112311002312 211013322233 122102311312 212030233311 130201312101,\\
$M_{17}$:&
121223023111 333301101022 211310330302 203113031032 022113301303 303030311201\\&
311102010013 132333222332 231033122213 122120133332 232012213111 130201112321,\\
$M_{18}$:&
301221023133 333101321222 211332130320 223133231010 222333121301 103232333023\\&
131322032031 112131022312 011211302231 102120313112 012012031113 330223312101,\\
$M_{19}$:&
101203003133 311103121200 231132112102 223333033010 202131323103 323032313023\\&
333320010211 330131200312 231033322011 300322333132 030232011331 110023330301,\\
$M_{20}$:&
323021021133 313301321020 033330112102 023331231012 222111301321 101210311003\\&
311102230033 330133222332 213213122211 302120113110 212012011313 132203132103,
\end{tabular}
\caption{New self-dual $\ZZ_4$-codes of length $24$ and $d_L=10$}
\label{Fig:24}
}
\end{figure}

\setcounter{figure}{0}
\begin{figure}[p]
\centering
{\footnotesize
\begin{tabular}{ll}
$M_{21}$:&
323203201311 133303121000 011112330120 003113213230 222111301321 123212313023\\&
113320232231 330331020110 211211102011 122100111332 030210033311 312201110301,\\
$M_{22}$:&
103023023311 131323103200 213132112102 201133011210 000111301321 301212113021\\&
113122010213 112113220130 031211300233 320120113112 032232013333 310203332101,\\
$M_{23}$:&
323221023131 111121303202 213132110120 221131233032 220111321321 123012113221\\&
311302212211 112131002130 031233120031 302102131110 230230213131 130201330103,\\
$M_{24}$:&
301001221131 133123321000 231110112300 023131213230 000313321321 103030311203\\&
331322032011 312331020312 011033120211 120100311312 230030031131 130201332321,\\
$M_{25}$:&
103223023131 331301123000 233312332322 021331013230 200331323101 101232331003\\&
111300012031 312331222310 213211302211 302300133112 010230231111 130201330321,\\
$M_{26}$:&
103221223133 313123321020 033132132120 021133013032 000331323303 321230131003\\&
331300210213 312113022112 211033102011 300300131312 030010211333 112023312101,\\
$M_{27}$:&
101003201311 313301321000 233132312102 203111013212 022133103123 121010311003\\&
111102012033 112331202110 011213120031 122320113330 012010211111 110021132321,\\
$M_{28}$:&
123023003333 331101101222 233130112320 201313233212 222111321301 323030331203\\&
333102232011 310113202110 213231120033 102102331332 010212033313 310021112123,\\
$M_{29}$:&
323203223111 131321303000 033312332322 023331031212 020111321321 121012331001\\&
311102010013 330113020312 031013122033 120100313110 230010031111 112001132103,\\
$M_{30}$:&
123201023133 113323301220 013130312322 201313233212 020333303301 301012133003\\&
311122012031 312331222132 031013320031 100300331312 032010233333 130223130321,\\
$M_{31}$:&
121223223331 131321323022 213332112322 221113213012 220133303323 303032333221\\&
133302012031 330333020132 011213320033 322302313330 210232233311 110021130103,\\
$M_{32}$:&
321201221113 333323321222 031312130320 003131033212 200113323321 101010333223\\&
113102012213 312331000130 031031322213 322100331132 012212031333 310223110121,\\
$M_{33}$:&
323201201133 313323101022 033330130302 203131013032 002133121323 123210311021\\&
313122232231 332111000110 033013322011 320320131310 030012031313 132201332321,\\
$M_{34}$:&
321023023131 111303303022 211112332120 021333231212 000311323123 103212333021\\&
131302010211 312333022332 011213322033 320122311110 210032213313 110221312301,\\
$M_{35}$:&
101023021333 133123301200 213332330322 201133231230 202333101301 323210311003\\&
131100032033 310131200112 011233320013 120300333110 230210231313 312001110321,\\
$M_{36}$:&
103023201133 113323321020 211312312120 221113231212 022111101121 303230133223\\&
133300012031 130313022310 213013120233 322320313112 210210213113 332001130103,\\
$M_{37}$:&
123003001313 331121303200 031312310322 203311033012 020131321323 323210111221\\&
333302210011 132131222112 213033120013 322102333112 210012231131 330023310101,\\
$M_{38}$:&
301201001311 131103323000 031332332322 201333233032 020113301323 123232133001\\&
331300230233 112333200130 011213102213 102320331312 032232031113 330221330323,\\
$M_{39}$:&
103023203311 311103121022 011312112322 001313013230 202131303123 123232331001\\&
111322030011 110313200110 033211120233 320322311130 230032231113 332223112301,\\
$M_{40}$:&
303223021311 331321103202 213332332302 021131233032 202333123303 301030331021\\&
113300010211 130333202332 211031120231 122120333312 210230233311 130221110321,\\
\end{tabular}
\caption{New self-dual $\ZZ_4$-codes of length $24$ and $d_L=10$ (continued)}
}
\end{figure}

\setcounter{figure}{0}
\begin{figure}[thb]
\centering
{\footnotesize
\begin{tabular}{ll}
$M_{41}$:&
101001021133 333323301222 031110132120 201333213012 020331121323 103012113201\\&
133120010033 330311020312 211213100213 122322113332 210230031313 332001132301,\\
$M_{42}$:&
323223023111 313121121202 011110130302 203111033232 020133123321 123012313203\\&
311102212231 130333202332 213011320033 322102113110 232030011133 332021330101,\\
$M_{43}$:&
321021201333 111321301020 011130330322 223113013230 202113323321 123012311221\\&
131320232013 310131020330 233231122231 302322133110 232010011333 112223312101,\\
$M_{44}$:&
121201221133 331103321002 211110312322 223131033010 002333303123 303210133221\\&
133120212233 310113022110 031011102013 122122111110 032032211133 132021130121,\\
$M_{45}$:&
323201023311 333101301002 213330332300 003311013230 222133303323 303210311203\\&
311102010013 130311222112 231031100233 120120331112 210210231333 312221112301,\\
$M_{46}$:&
321221201333 333303123000 011130332322 201333233032 000111103321 123230113203\\&
113302030031 312313020130 011213122031 102300313130 210012031333 110203332301,\\
$M_{47}$:&
321001023313 111123103222 011312312100 201133011010 202313301321 101030113201\\&
313100232013 312311220112 211231322011 102100111112 210010211133 132201110321,\\
$M_{48}$:&
103023223311 333103101002 233112332122 003311011212 200131301101 323012133201\\&
113322032031 112333202132 213213302213 300122133310 012230211333 132023132103,\\
$M_{49}$:&
323021223333 111123103222 031110310120 223131231030 222333323301 321232331021\\&
113322230213 112313022330 031033120013 120102311130 032232031333 310203332323,\\
$M_{50}$:&
121223003313 131301303020 213332132102 023131231232 022111123123 123012113001\\&
331120012233 310331220132 211013320011 100120313110 230212011113 310023130123,\\
$M_{51}$:&
103023223113 113123303222 233130310102 023131231230 000131303321 103212313003\\&
131300012011 312111000132 231033102213 320120111132 012030013313 310221312303,\\
$M_{52}$:&
103221003331 113323103220 213112132102 223311011210 002111101121 321210131001\\&
111320010013 330133222112 013013302231 122122113130 012010031311 310201110321,\\
$M_{53}$:&
303221221313 313303303022 033330130302 023311233212 002311123323 121012133001\\&
333122230033 110131202130 211211122033 300322333312 030032031113 312001332321,\\
$M_{54}$:&
103023221133 313103301202 011312330302 023331013210 220331103123 121210111201\\&
313122030031 132111220312 013213320231 120100313110 010032011113 310001332321,\\
$M_{55}$:&
303023003313 111301123002 233110310320 201111233010 202133121303 323010131003\\&
311300210033 110111200330 211033102211 120302333112 212232213331 110201310123,\\
$M_{56}$:&
103023203113 133123123222 031310330122 001133231030 002111303323 123012333221\\&
333100232233 110113202132 233011320013 320102113332 010210231333 330003312101,\\
$M_{57}$:&
321003003131 311323123002 211310112320 223333233210 222133123321 103230113203\\&
313300210031 132333222112 031033122031 122120313332 030212233333 112203132303 
\end{tabular}
\caption{New self-dual $\ZZ_4$-codes of length $24$ and $d_L=10$ (continued)}
}
\end{figure}

\end{document}